\documentclass[proceedings,submission,notimes]{fpsacarxiv}

\usepackage[latin1]{inputenc}
\usepackage{subfigure}

\usepackage[round]{natbib}
\usepackage{colordvi}
\usepackage{bbm,bm}
\usepackage{amssymb}
\usepackage{amsmath}

\numberwithin{equation}{section}
\newtheorem{theorem}{Theorem}[section]
\newtheorem{proposition}[theorem]{Proposition}
\newtheorem{lemma}[theorem]{Lemma}
\newtheorem{corollary}[theorem]{Corollary}
\newtheorem{defi}[theorem]{Definition}
\newtheorem{exam}[theorem]{Example}
\newtheorem{rema}[theorem]{Remark}
\newenvironment{example}{\rm\begin{exam}\rm}{\end{exam}}
\newenvironment{remark}{\rm\begin{rema}\rm}{\end{rema}}

\newcommand{\demph}[1]{\emph{\Blue{#1}}}

\def\Cofree{\mathsf{C}}

\newcommand{\calC}{\mathcal{C}}
\newcommand{\calD}{\mathcal{D}}
\newcommand{\calE}{\mathcal{E}}

\newcommand{\K}{\mathbb{K}}

\def\c{\mathfrak C}
\def\ck{{\mathcal C}{\mathcal K}}
\def\m{\mathcal M}
\def\p{\mathcal P}
\def\s{\mathfrak S}
\def\y{\mathcal Y}

\def\csym{\c\mathit{Sym}}
\def\cksym{\ck\mathit{Sym}}
\def\msym{\m\mathit{Sym}}
\def\psym{\mathcal PSym}
\def\ssym{\s\mathit{Sym}}
\def\ysym{\y\mathit{Sym}}

\def\btau{\bm\tau}
\def\bkappa{\bm\kappa}
\def\bvarphi{\bm\varphi}

\def\bb{\boldbullet}
    \def\boldbullet{{\!\mbox{\LARGE$\mathbf\cdot$}}}
\newcommand{\Placeholder}{\raisebox{.09ex}{\footnotesize $\bullet$}}

\newcommand{\sn}[1]{{\small\hspace{2pt}$#1$}}
\def\id{\mathbbm{1}} %\mathrm{id}}
\def\psplit{\stackrel{\curlyvee}{\to}}

\newcommand{\CTO}[1]{\begin{picture}(4,11)\put(0,0){\includegraphics{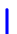}}
    \put(0.6,8.5){\tiny #1}\end{picture}}
\newcommand{\CTI}[2]{\begin{picture}(11.5,12)(-1,0)\put(1.5,0){\includegraphics{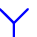}}
    \put(0,9){\tiny #1} \put(8,9){\tiny #2} \end{picture}}
\newcommand{\CTIT}[3]{\begin{picture}(15,13)(-1.2,0)\put(1.5,0){\includegraphics{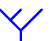}}
    \put(0,10){\tiny #1} \put(6,10){\tiny #2}\put(12,10){\tiny #3} \end{picture}}

\def\treeseparator{\bm\cdot}
\def\treeseparatorinline{\bm\cdot}
\newcommand{\compose}[3]{\frac{{#1{\treeseparator}\dotsb{\treeseparator}#2}}{#3}}
\newcommand{\lcompose}[4]{\frac{{#1{\treeseparator}#2{\treeseparator}\dotsb{\treeseparator}#3}}{#4}}
\newcommand{\rcompose}[4]{\frac{{#1{\treeseparator}\dotsb{\treeseparator}#2{\treeseparator}#3}}{#4}}
\newcommand{\zcompose}[2]{\frac{{#1}}{#2}}
\newcommand{\lrcompose}[5]{\frac{{#1{\treeseparator}#2{\treeseparator}\dotsb{\treeseparator}#3{\treeseparator}#4}}{#5}}

\newcommand{\composeinline}[3]{{#3{\,\circ\,}({#1{\treeseparatorinline}\dotsb{\treeseparatorinline}#2})}}
\newcommand{\zcomposeinline}[2]{{{#2}{\,\circ\,}(#1)}}

\author{Stefan Forcey~\addressmark{1}\!,
        Aaron Lauve\addressmark{2}\thanks{Supported in part by NSA grant
          H98230-10-1-0362.}\,, \and
        Frank Sottile\addressmark{3}\thanks{Supported in part by NSF grants DMS-0701050 and DMS-1001615.}}
\title[Cofree  compositions of  coalgebras]{Cofree compositions of coalgebras\\(extended abstract)}
\address{\addressmark{1}Department of Theoretical and Applied Mathematics,
    The University of Akron, Akron, OH 44325 USA\\
\addressmark{2}Department of Mathematics,
    Loyola University of Chicago,
    Chicago, IL 60660   USA\\
\addressmark{3}Department of Mathematics,
    Texas A\&M University,
    College Station, TX 77843  USA}
\keywords{multiplihedron, cofree coalgebra, Hopf algebra, operad, species}
\begin{document}
\maketitle
%%%%%%%%%%%%%%%%%%%%%%%%%%%%%%%%%%%%%%%%%%%%%%%%%%%%%%%%%
\begin{abstract}
\paragraph{Abstract.}
  We develop the notion of the composition of two coalgebras,
 which arises naturally in higher category theory and the theory of species.
 We prove that the composition of two cofree coalgebras is cofree
 and give conditions which imply that the composition is a one-sided Hopf algebra.
 These conditions hold when
 one coalgebra is a graded Hopf operad $\calD$ and the other is a connected graded coalgebra
 with coalgebra map to $\calD$.
 We conclude by discussing these structures for compositions with bases
 the vertices of multiplihedra, composihedra, and hypercubes.

%\paragraph{R\'esum\'e.}

\end{abstract}
%%%%%%%%%%%%%%%%%%%%%%%%%%%%%%%%%%%%%%%%%%%%%%%%%%%%%%%%%

%-------------------------------------------------------------------
% Begin SECTION
%-------------------------------------------------------------------
\section{Introduction}\label{sec: intro}

The Hopf algebras of ordered trees~(\cite{MalReu:1995}) and of
planar binary trees~(\cite{LodRon:1998}) are cofree coalgebras that
are connected by cellular maps from permutahedra to associahedra.
Related polytopes include the multiplihedra~(\cite{Sta:1970}) and
the composihedra~(\cite{forcey2}), and it is natural to study what
Hopf structures may be placed on these objects. The map from
permutahedra to associahedra factors through the multiplihedra, and
in~(\cite{FLS:1}) we used this factorization to place Hopf
structures on bi-leveled trees, which correspond to vertices of
multiplihedra.

Multiplihedra form an operad module over associahedra. This leads to
painted trees, which also correspond to the vertices of the
multiplihedra. In terms of painted trees, the Hopf structures
of~(\cite{FLS:1}) are related to the operad module structure. We
generalize this in Section~\ref{sec: cccc}, defining the functorial
construction of a graded coalgebra $\calD\circ\calC$ from graded
coalgebras $\calC$ and $\calD$. We show that this composition of
coalgebras preserves cofreeness. In Section \ref{sec: hopf},  give
sufficient conditions when $\calD$ is a Hopf algebra for the
composition of coalgebras $\calD\circ\calC$ (and $\calC\circ\calD$)
to be a one-sided Hopf algebra. These conditions also guarantee that
a composition is a Hopf module and a comodule algebra over $\calD$.

This composition arises in the theory of operads and in the theory
of species, as species and operads are one-and-the-same~\cite[App.
B]{AguMah:2010}. In Section \ref{sec: hopf} we show that an operad
$\calD$ of connected graded coalgebras is automatically a Hopf
algebra.

 We discuss three examples related to well-known objects from
category theory and algebraic topology and show that the Hopf
algebra of simplices of~(\cite{ForSpr:2010}) is cofree as a
coalgebra.
%%%%%%%%%%%%%%%%%%%%%%%%%%%%%%%%%%%%%%%%%%%%%%%%%%%%%%%%%

%-------------------------------------------------------------------
% Begin SECTION
%-------------------------------------------------------------------
\section{Preliminaries}\label{sec: prelims}

We work over a fixed field $\K$ of characteristic zero. For a graded
vector space $V = \bigoplus_n V_n$, we write $|v|=n$ and say $v$ has
\demph{degree} $n$ if $v\in V_n$.

%-------------------------------------------------------------------
\subsection{Hopf algebras and cofree coalgebras}\label{sec: cofree def}

A Hopf algebra $H$ is a unital associative algebra equipped with a
coassociative coproduct homomorphism $\Delta\colon H\to  H\otimes H$
and a counit homorphism $\varepsilon\colon H\to \K$ which plays the
role of the identity for $\Delta$. See~(\cite{Mont:1993}) for more
details. \cite{Tak:71} showed that a graded bialgebra
$H=(\bigoplus_{n\geq0}H_n,\bm\cdot,\Delta,\varepsilon)$ that is
connected ($H_0 = \K$) is a Hopf algebra.

A coalgebra $\calC$ is a vector space $\calC$ equipped with a
coassociative coproduct $\Delta$ and counit $\varepsilon$. For
$c\in\calC$, write  $\Delta(c)$ as $\sum_{(c)} c'\otimes c''$.
Coassociativity means that
\[
    \sum_{(c),(c')} (c')'\otimes (c')'' \otimes c''\ =\
    \sum_{(c),(c'')} c'\otimes (c'')' \otimes (c'')''\ =\
    \sum_{(c)} c'\otimes c'' \otimes c''' \,,
\]
and the counit condition means that $\sum_{(c)} \varepsilon(c')c'' =
\sum_{(c)} c'\varepsilon(c'') = c$.

The \demph{cofree coalgebra} on a vector space $V$ is $\Cofree(V):=
\bigoplus_{n\geq0}V^{\otimes n}$ with counit the projection
$\varepsilon\colon\Cofree(V)\to\K=V^{\otimes 0}$ and the
\demph{deconcatenation coproduct}: writing ``$\backslash$'' for the
tensor product in $V^{\otimes n}$, we have
\[
    \Delta(c_1\backslash \dotsb\backslash c_n)\ =\
    \sum_{i=0}^{n} (c_1\backslash \dotsb\backslash c_{i})
        \otimes (c_{i+1}\backslash \dotsb\backslash c_n) \,.
\]
Observe that $V$ is the set of primitive elements of $\Cofree(V)$. A
coalgebra $\calC$ is \demph{cofree} if
$\calC\simeq\Cofree(P_\calC)$, where $P_\calC$ is the space of
primitive elements of $\calC$. Many coalgebras arising in
combinatorics are cofree.

%-------------------------------------------------------------------
\subsection{Cofree Hopf algebras on trees}\label{sec: Hopf examples}

We describe three cofree Hopf algebras built on rooted planar binary
trees: \demph{ordered trees} $\s_n$, \demph{binary trees} $\y_n$,
and \demph{(left) combs} $\c_n$ on $n$ internal nodes. Set $\s_\bb
:=\bigcup_{n\geq0} \s_n$ and define $\y_\bb$ and $\c_\bb$ similarly.

%-------------------------------------------------------------------
\subsubsection{Constructions on trees}\label{sec: trees}

The nodes of a tree $t\in\y_n$ form a poset. An \demph{ordered tree
$w=w(t)$} is a linear extension of this node poset of $t$. This
linear extension is indicated by placing a permutation in the gaps
between its leaves, which gives a bijection between ordered trees
and permutations. The map $\tau \colon \s_n \to \y_n$ sends an
ordered tree $w(t)$ to its underlying tree $t$. The map $\kappa
\colon \y_n \to \c_n$ shifts all nodes of a tree to the right branch
from the root. Set $\s_0=\y_0=\c_0=\includegraphics{0.eps}$. Note
that $|\c_n|=1$ for all $n\geq0$.
%%%%%%%%%%%%%%%%%%%%%%%%%%%%%%%%%%%%%%%%%%%%%%%%%%%%%%%%%%%%%%%%%%%%%%%%%
%\begin{figure}[htb]
\[
\begin{picture}(85,83)(0,5)
%   \showgrid
   \thicklines
   \put(11,4){\small ordered trees $\s_\bb$}
   \put(0,55){%
      \begin{picture}(40,35)
         \put(0,0){\includegraphics{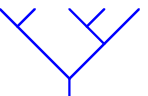}}
      \put(0,27){\sn{3}} \put(10,27){\sn{4}} \put(20,27){\sn{1}} \put(30,27){\sn{2}}
      \end{picture}}
   \put(45,35){%
      \begin{picture}(40,35)
         \put(0,0){\includegraphics{3412.d.eps}}
      \put(0,27){\sn{2}} \put(10,27){\sn{4}} \put(20,27){\sn{1}} \put(30,27){\sn{3}}
      \end{picture}}
   \put(0,15){%
      \begin{picture}(40,35)
         \put(0,0){\includegraphics{3412.d.eps}}
      \put(0,27){\sn{1}} \put(10,27){\sn{4}} \put(20,27){\sn{2}} \put(30,27){\sn{3}}
      \end{picture}}

\end{picture}
\qquad
\begin{picture}(30,80)
   \thicklines
   \put(0,4){\vector(1,0){30}}\put(10,8){$\tau$}

   \put(0, 42){\Color{1 0 1 .3}{\vector(1,0){30}}}
\end{picture}
\qquad
\begin{picture}(65,80)
   \thicklines
   \put(8,4){\small binary trees $\y_\bb$}

   \put(15,30){%
      \includegraphics{3412.d.eps}}

\end{picture}
\qquad
\begin{picture}(30,80)
   \thicklines
   \put(0,4){\vector(1,0){30}}\put(10,8){$\kappa$}

   \put(0, 42){\Color{1 0 1 .3}{\vector(1,0){30}}}
\end{picture}
\qquad
\begin{picture}(65,80)
   \thicklines
   \put(5,4){\small left combs $\c_\bb$}
   \put(5,30){\includegraphics{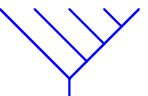}}
\end{picture}
\]
%%%%%%%%%%%%%%%%%%%%%%%%%%%%%%%%%%%%%%%%%%%%%%%%%%%%%%%%%%%%%%%%%%%%%%%%%

\demph{Splitting} an ordered tree $w$ along the path from a leaf to
the root yields an ordered forest (where the nodes in the forest are
totally ordered) or a pair of ordered trees,
\[
  \begin{picture}(52,42)
   \put(0,0){\includegraphics{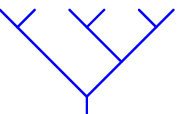}}
   \put(0,31){\sn{2}} \put(10,31){\sn{5}} \put(20,31){\sn{1}}
   \put(30,31){\sn{4}}\put(40,31){\sn{3}}
   \put(30,44){\vector(0,-1){10}}
  \end{picture}
   \ \raisebox{12pt}{\large$\leadsto$}\
  \begin{picture}(52,42)
   \put(0,0){\includegraphics{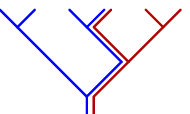}}
   \put(0,31){\sn{2}} \put(10,31){\sn{5}} \put(20,31){\sn{1}}
   \put(32,31){\sn{4}}\put(42,31){\sn{3}}
   \put(31,44){\vector(0,-1){10}}
  \end{picture}
   \ \raisebox{12pt}{$\xrightarrow{\ \curlyvee\ }$}\quad
 \raisebox{7.5pt}{$\displaystyle
   \raisebox{8pt}{$\biggl(\;$}
   \begin{picture}(32,27)
    \put(0,0){\includegraphics{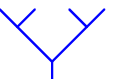}}
    \put(0,21){\sn{2}} \put(10,21){\sn{5}} \put(20,21){\sn{1}}
   \end{picture}
     ,\
   \begin{picture}(22,23)
    \put(0,0){\includegraphics{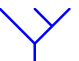}}
    \put(0,16){\sn{4}} \put(10,16){\sn{3}}
   \end{picture}
   \raisebox{8pt}{$\;\biggr)$}
   \quad \raisebox{5.5pt}{\mbox{or}}\quad
   \raisebox{8pt}{$\biggl(\;$}
   \begin{picture}(32,27)
    \put(0,0){\includegraphics{231.d.eps}}
    \put(0,21){\sn{2}} \put(10,21){\sn{3}} \put(20,21){\sn{1}}
   \end{picture}
     ,\
   \begin{picture}(22,23)
    \put(0,0){\includegraphics{21.d.eps}}
    \put(0,16){\sn{2}} \put(10,16){\sn{1}}
   \end{picture}
   \raisebox{8pt}{$\;\biggr)$}.
    $}
\]
Write $w\psplit(w_0,w_1)$ when the ordered forest $(w_0,w_1)$ (or
pair of ordered trees) is obtained by splitting $w$. (Context will
determine how to interpret the result.)

We may \demph{graft} an ordered forest $\vec w = (w_0,\dotsc,w_n)$
onto an ordered tree $v\in\s_n$, obtaining the tree $\vec w/v$ as
follows. First increase each label of $v$ so that its nodes are
greater than the nodes of $\vec w$, and then graft tree $w_i$ onto
the $i^{\mathrm{th}}$ leaf of $v$. For example,
\begin{align*}
\hbox{if } (\vec w, v)  &\ =\
  \raisebox{0pt}{$\displaystyle
   \raisebox{8pt}{$\Biggl( \biggl(\;$}
       \begin{picture}(22,23)
          \put(0,0){\includegraphics{21.d.eps}}
          \put(0,16){\sn{3}} \put(10,16){\sn{2}}
       \end{picture}
         , \
         \includegraphics{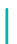}
        \ ,
       \begin{picture}(32,25)
          \put(0,0){\includegraphics{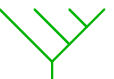}}
              \put(0,21){\sn{7}} \put(10,21){\sn{5}}\put(20,21){\sn{1}}
       \end{picture}
         , \
       \begin{picture}(12,18)
           \put(0,0){\includegraphics{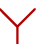}}
           \put(0,11){\sn{6}}
       \end{picture}
        \ , \
       \begin{picture}(12,18)
           \put(0,0){\includegraphics{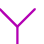}}
           \put(0,11){\sn{4}}
       \end{picture}
   \raisebox{8pt}{$\;\biggr)$} ,
       \begin{picture}(42,28)(-2,0)%\showgrid
              \put(0,0){\includegraphics{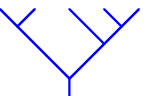}}
              \put(0,27){\sn{1}} \put(10,27){\sn{4}} \put(20,27){\sn{3}} \put(30,27){\sn{2}}
            \end{picture}
   \raisebox{8pt}{$\;\Biggr)$},
   $}
\\[1.5ex]
\raisebox{30pt}{then $\vec w/v$} & \ \  \raisebox{30pt}{$=$} \ \
  \begin{picture}(125,88)
    \put(0,0){\includegraphics{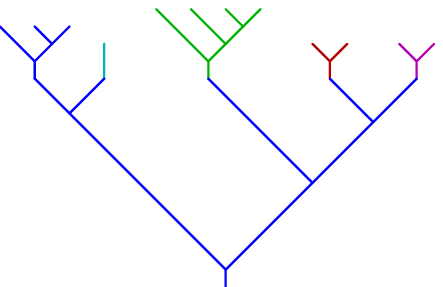}}
    \put(0,76){\sn{3}} \put(10,76){\sn{2}} \put(15,61){\sn{8}} \put(36,61){\sn{11}}
    \put(45,81){\sn{7}} \put(55,81){\sn{5}} \put(65,81){\sn{1}} \put(71,61){\sn{10}}
    \put(90,71){\sn{6}} \put(103,61){\sn{9}} \put(115,71){\sn{4}}
  \end{picture}
  \quad\raisebox{30pt}{$=$}\quad
  \begin{picture}(120,73)
    \put(0,0){\includegraphics{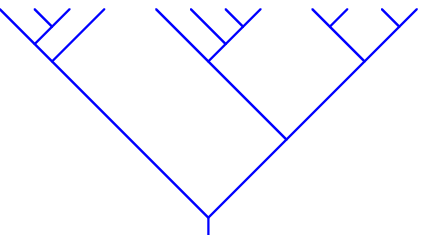}}
    \put(0,66){\sn{3}} \put(10,66){\sn{2}} \put(20,66){\sn{8}} \put(30,66){\sn{11}}
    \put(45,66){\sn{7}} \put(55,66){\sn{5}} \put(65,66){\sn{1}} \put(75,66){\sn{10}}
    \put(90,66){\sn{6}} \put(101,66){\sn{9}} \put(110,66){\sn{4}}
  \end{picture}
\raisebox{30pt}{.}
\end{align*}

Splitting and grafting make sense for trees in $\y_\bb$. They also
work for $\c_\bb$ if, after grafting a forest of combs onto the
leaves of a comb, one appliees $\kappa$ to the resulting planar
binary tree to get a new comb.

%-------------------------------------------------------------------
\subsubsection{Three cofree Hopf algebras}\label{sec: cofree examples}

Let $\ssym := \bigoplus_{n\geq0} \ssym_n$ be the graded vector space
whose $n^{\mathrm{th}}$ graded piece has basis $\{F_w \mid w \in
\s_n\}$. Define $\ysym$ and $\csym$ similarly. The set maps $\tau$
and $\kappa$ induce vector space maps $\btau$ and $\bkappa$,
$\btau(F_w) = F_{\tau(w)}$ and $\bkappa(F_t)=F_{\kappa(t)}$. Fix
$\mathfrak X \in \{\s, \y, \c\}$. For $w\in \mathfrak X_\bb$ and
$v\in \mathfrak X_n$, set\vspace{-1pt}
 \begin{align*} %\label{eq: F coproduct}
    F_w \cdot F_v\ &:=\ \sum_{w\psplit(w_0,\dotsc,w_n)} F_{(w_0,\dotsc,w_n)/v} \,,\vspace{-1pt}
  \intertext{the sum over all ordered forests obtained by splitting $w$ at a multiset of $n$
  leaves. For $w\in\mathfrak{X}_\bb$, set\vspace{-1pt} }
    \Delta(F_w) \ &:= \ \sum_{w\psplit(w_0,w_1)} F_{w_0} \otimes F_{w_1} \,,\vspace{-1pt}
 \end{align*}
the sum over all splittings of $w$ at one leaf. The counit
$\varepsilon$ is the projection onto the $0^\mathrm{th}$ graded
piece, spanned by the unit element $1=F_{\;\includegraphics{0.eps}}$
for the multiplication.

\begin{proposition}\label{thm: Hopf examples}
For $(\Delta,\cdot,\varepsilon)$ above,
 $\ssym$ is the Malvenuto--Reutenauer Hopf algebra of permutations,
 $\ysym$ is the Loday--Ronco Hopf algebra of planar binary trees, and
 $\csym$ is the divided power Hopf algebra.
Moreover, $\ssym\xrightarrow{\,\btau\,}\ysym$ and
$\ysym\xrightarrow{\,\bkappa\,}\csym$ are surjective Hopf algebra
maps. \hfill \qed
\end{proposition}

The part of the proposition involving $\ssym$ and  $\ysym$ is found
in~(\cite{AguSot:2005,AguSot:2006}); the part involving $\csym$ is
straightforward and we leave it to the reader.

  Typically~\cite[Ex~5.6.8]{Mont:1993}, the divided power Hopf algebra is defined to be
 $\K[x] := \mathrm{span}\{x^{(n)} \mid n\geq0\}$,
 with basis vectors $x^{(n)}$ satisfying
  $x^{(m)}\cdot x^{(n)} = \binom{m+n}{n} x^{(m+n)}$,
  $1=x^{(0)}$, $\Delta(x^{(n)}) = \sum_{i+j=n} x^{(i)} \otimes x^{(j)}$,
 and $\varepsilon(x^{(n)})=0$ for $n>0$.
 An  isomorphism between $\K[x]$ and $\csym$ is given by
  $x^{(n)}\mapsto F_{c_n}$, where $c_n$ is the
 unique comb in $\c_n$.

%The following result is important for what follows.

\begin{proposition}\label{thm: cofree examples}
 The Hopf algebras $\ssym$, $\ysym$, and $\csym$ are cofree as coalgebras.
 The primitive elements of $\ysym$ and $\csym$ are indexed by trees with no nodes
 off the right branch from the root.
\hfill \qed
\end{proposition}

The result for $\csym$ is easy. Proposition~\ref{thm: cofree
examples} is proven for $\ssym$ and $\ysym$
in~(\cite{AguSot:2005,AguSot:2006}) by performing a change of
basis---from the \demph{fundamental basis} $F_w$ to the
\demph{monomial basis} $M_w$---by means of M\"obius inversion in a
poset structure placed on $\s_\bb$ and $\y_\bb$.
%%%%%%%%%%%%%%%%%%%%%%%%%%%%%%%%%%%%%%%%%%%%%%%%%%%%%%%%%

%-------------------------------------------------------------------
% Begin SECTION
%-------------------------------------------------------------------
\section{Constructing Cofree Compositions of Coalgebras}\label{sec: cccc}

%-------------------------------------------------------------------
\subsection{Cofree compositions of coalgebras}\label{sec: cccc main results}

Let $\calC$ and $\calD$ be graded coalgebras. Form a new coalgebra
$\calE=\calD\circ\calC$ on the vector space
 \begin{gather}\label{eq: E_(n)}
  \calD\circ\calC \ := \ \bigoplus_{n\geq0} \calD_n \otimes \calC^{\otimes(n+1)} \,.
 \end{gather}
When $\calC$ and $\calD$ are spaces of rooted, planar trees we may
interpret $\circ$ in terms of a rule for grafting trees.

%%%%%%%%%%%%%%%%%%%%%%%%%%%%%%%%%%%%%%%%%%%%%%%%%%%%%%%%%%%%%%%%%%%%%
\begin{example}\label{ex: painted}
 Suppose $\calC=\calD=\ysym$ and let
 $d\times (c_0,\dotsc,c_n) \in \Red{\y_n} \times \bigl(\Blue{\y_\bb}^{{n+1}}\bigr)$.
 Define $\circ$ by attaching the forest $(c_0,\dots,c_n)$ to the leaves of $d$
 while remembering $d$,
 \[
   \raisebox{0pt}{$\displaystyle
      \includegraphics{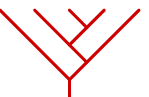}
\raisebox{2pt}{\ \ $\times$\ }
    \raisebox{8pt}{$\biggl(\;$}
      \includegraphics{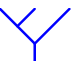}
     , \
      \includegraphics{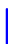}
      \ ,
      \includegraphics{231.d.eps}
     , \
      \includegraphics{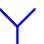}
         \ , \
      \includegraphics{0.d.eps}
   \raisebox{8pt}{$\;\biggr)$}
   $} \raisebox{3pt}{\ \ \ $\stackrel{\textstyle\circ}{\longmapsto}$\ }
   \raisebox{-18pt}{\includegraphics{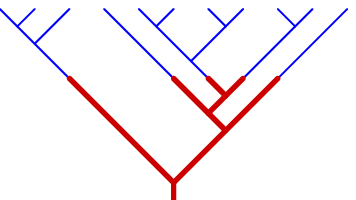}}
 \]
% This gives a new type of tree (\demph{painted trees} in Section \ref{sec:Examples}).
%% Applying this to the indices of basis elements and extending by multilinearity to
% $\calD \circ \calC$ gives the structure we seek.
\end{example}

We represent an indecomposable tensor in $\calE:=\calD\circ\calC$ as
\[
    \composeinline{c_0}{c_n}{d} \qquad\hbox{or}\qquad \compose{c_0}{c_n}{d}\ .
\]
The \demph{degree} of such an element is $|d|+|c_0|+\dotsb+|c_n|$.
Write $\calE_n$ for the span of elements of degree $n$.

%-------------------------------------------------------------------
\subsubsection{The coalgebra $\calD\circ\calC$}
We define the \demph{compositional coproduct} $\Delta$ for
$\calD\circ\calC$ on indecomposable tensors: if $|d|=n$,
put\vspace{-2pt}
 \begin{equation}\label{eq: delta}
    \Delta\left(\compose{c_0}{c_n}{d}\right)\ =\
      \sum_{i=0}^n \sum_{\substack{(d) \\ |d'|=i}}
    \sum_{(c_i)} \rcompose{c_0}{c_{i-1}}{c'_i}{d'} \otimes \lcompose{c''_i}{c_{i+1}}{c_n}{d''} \,.\vspace{-2pt}
 \end{equation}

The \demph{counit} $\varepsilon:\calD\circ\calC \to \K$ is given by
$\varepsilon(\composeinline{c_0}{c_n}{d})=
  \varepsilon_\calD(d) \cdot \prod_j \varepsilon_\calC(c_j) $.

For the painted trees of Example~\ref{ex: painted}, if the $c_i$ and
$d$ are elements of the $F$-basis, then
 $\Delta ({\composeinline{c_0}{c_n}{d}})$ is the sum over all splittings $t\psplit(t',t'')$ of
 $t$ into a pair of painted trees.

\begin{theorem}
  $(\calD\circ\calC,\Delta,\varepsilon)$ is  a coalgebra.
  This composition is functorial, if $\varphi\colon\calC\to\calC'$ and
  $\psi\colon\calD\to\calD'$ are morphisms of graded coalgebras, then\vspace{-3pt}
\[
    \compose{c_0}{c_n}{d}\ \longmapsto\
    \compose{\varphi(c_0)}{\varphi(c_n)}{\psi(d)}\vspace{-3pt}
\]
  defines a morphism of graded coalgebras
  $\varphi\circ\psi\colon \calD\circ\calC\to \calD'\circ\calC'$.
\end{theorem}
%%%%%%%%%%%%%%%%%%%%%%%%%%%%%%%%%%%%%%%%%%%%%%%%%%%%%%%%%%%%%%%%%%%%%%%%%%%%%

%-------------------------------------------------------------------
\subsubsection{The cofree coalgebra $\calD\circ\calC$}
Suppose that $\calC$ and $\calD$ are graded, connected, and cofree.
Then $\calC=\Cofree(P_\calC)$, where $P_\calC\subset\calC$ is its
space of primitive elements. Likewise, $\calD=\Cofree(P_\calD)$,
where $P_\calD\subset\calD$ is its space of primitive elements.
%As in Section~\ref{sec: cofree def}, we use ``$\backslash$''  for internal tensor products.

%%%%%%%%%%%%%%%%%%%%%%%%%%%%%%%%%%%%%%%%%%%%%%%%%%%%%%%%%%%%%%%%%%%%%%%%%%%
\begin{theorem}\label{thm: cc cofree}
 If\/ $\calC$ and $\calD$ are cofree coalgebras  then $\calD\circ\calC$ is also a cofree
 coalgebra.
 Its space of  primitive  elements is spanned by indecomposible tensors of the form\vspace{-2pt}
\begin{equation}\label{eq: primitives}
    \lrcompose{1}{c_1}{c_{n-1}}{1}{\delta} \qquad\hbox{ and }\qquad
    \zcompose{\,\gamma\,}{1} \,,\vspace{-2pt}
\end{equation}
 where $\gamma,c_i\in\calC$ and $\delta\in\calD_n$ with $\gamma,\delta$ primitive.
\end{theorem}

%-------------------------------------------------------------------
\begin{example}\label{sec: cccc examples}

The graded Hopf algebras of ordered trees $\ssym$, planar trees
$\ysym$, and divided powers $\csym$ are all cofree, and so their
compositions are cofree. We have the surjective Hopf algebra
maps\vspace{-2pt}
\[
   \ssym\ \xrightarrow{\ \btau\ }\ \ysym\ \xrightarrow{\ \bkappa\ }\ \csym\vspace{-2pt}
\]
giving the commutative diagram of Figure~\ref{fig: diamond} of nine
cofree coalgebras as the composition $\circ$ is functorial.
%Except for $\ssym\circ\ssym$, this is a commutative diagram of one-sided Hopf
%algebras.
%We discuss the underlined algebras in Section \ref{sec:Examples}.
 %%
\begin{figure}[!hbt]
\[
  \begin{picture}(270,160)(0,5)
                       \put(105,160){$\ssym\circ\ssym$}
    \put(128,154){\vector(-1,-1){20}}     \put(146,154){\vector(1,-1){20}}
       \put(65,120){$\ssym\circ\ysym$}   \put(145,120){$\ysym\circ\ssym$}
    \put(106,114){\vector(1,-1){20}}     \put(168,114){\vector(-1,-1){20}}
    \put( 88,114){\vector(-1,-1){20}}      \put(186,114){\vector(1,-1){20}}
   \put(25,81){$\ssym\circ\csym$}        \put(185,81){$\csym\circ\ssym$}
                    \put(105,81){\Blue{\underline{$\ysym\circ\ysym$}}}
    \put( 66,74){\vector(1,-1){20}}     \put(128,74){\vector(-1,-1){20}}
    \put(146,74){\vector(1,-1){20}}     \put(208,74){\vector(-1,-1){20}}
      \put(65,41){\Blue{\underline{$\ysym\circ\csym$}}}   \put(145,41){$\csym\circ\ysym$}
    \put(106,34){\vector(1,-1){20}}     \put(168,34){\vector(-1,-1){20}}
                  \put(105,3){\Blue{\underline{$\csym\circ\csym$}}}
  \end{picture}
\]
\caption{A commutative diagram of cofree compositions of
coalgebras.}\label{fig: diamond}
\end{figure}
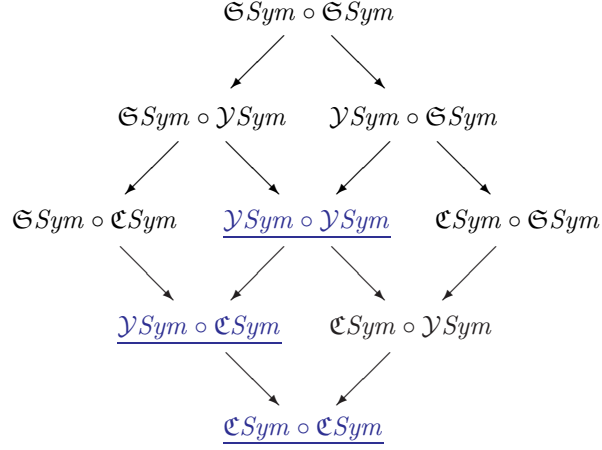
\end{example}
%-------------------------------------------------------------------
\subsection{Enumeration}
%We enumerate the graded dimension of many examples from Example~\ref{sec: cccc examples}.
Set $\calE:=\calD\circ\calC$ and let $C_n$ and $E_n$ be the
dimensions of $\calC_n$ and $\calE_n$, respectively.

\begin{theorem}
 When $\calD_n$ has a basis indexed by combs with $n$ internal nodes
 we have the recursion \vspace{-2pt}
 \[
   E_0\ =\ 1\,,\quad\mbox{and for \ $n>0$,}\quad
  E_n\ =\ C_n\ +\ \sum_{i=0}^{n-1}C_i E_{n-i-1}\,.\vspace{-2pt}
 \]
\end{theorem}
\begin{proof}
 The first term counts elements  in $\calE_n$ of the form
 $\includegraphics{0.eps}\circ c$.
 Removing the root node of $d$ from ${\composeinline{c_0}{c_k}{d}}$ gives a
 pair  $\zcomposeinline{c_0}{\includegraphics{0.eps}}$ and
 $\composeinline{c_1}{c_k}{d'}$ with $c_0\in\calC_i$,
 whose dimensions are enumerated by the terms of the sum.
\end{proof}

For  combs over a comb, $E_n=2^n$, for trees over a comb, $E_n$ are
the Catalan numbers, and for permutations over a comb, we have the
recursion
$$
   E_0\ =\ 1\,,\quad\mbox{and for \ $n>0$,}\quad
   E_n\ =\ n! + \sum_{i=0}^{n-1}i! E_{n-i-1}\,,
$$
 which begins $1, 2, 5, 15, 54, 235,\dotsc$, and is sequence A051295 of the OEIS~(\cite{Slo:oeis}).
Similarly,

\begin{theorem}
When $\calD_n$ has a basis indexed by $\y_n$ then we have the
recursion
 $$
   E_0\ =\ 1\,,\quad\mbox{and for \ $n>0$,}\quad
   E_n\ =\ C_n + \sum_{i=0}^{n-1}E_i E_{n-i-1} \,.
 $$
\end{theorem}

For example, the combs over a tree are enumerated by the binary
transform of the Catalan numbers (\cite{forcey2}). The trees over a
tree are enumerated by the Catalan transform of the Catalan numbers
(\cite{forcey1}). The permutations over a tree are enumerated by the
recursion
$$
   E_0\ =\ 1\,,\quad\mbox{and for \ $n>0$,}\quad
   E_n\ =\ n! + \sum_{i=0}^{n-1}E_i E_{n-i-1}\,,
$$
which begins $1, 2, 6, 22, 92, 428,\dotsc$  and is not a recognized
sequence in the OEIS~(\cite{Slo:oeis}).

%%%%%%%%%%%%%%%%%%%%%%%%%%%%%%%%%%%%%%%%%%%%%%%%%%%%%%%%%
%---------------------------------------------------------------------------
% Begin SECTION
%---------------------------------------------------------------------------
\section{Composition of Coalgebras and Hopf Modules}\label{sec: hopf}
%---------------------------------------------------------------------------

We give conditions that imply a composition of coalgebras is a
one-sided Hopf algebra, interpret this via operads, and then
investigate which compositions of Fig.~\ref{fig: diamond} are
one-sided Hopf algebras.

%%%%%%%%%%%%%%%%%%%%%%%%%%%%%%%%%%%%%%%%%%%%%%%%%%%%%%%%%
%
\subsection{Module coalgebras}\label{sec: covering}

Let $\calD$ be a connected graded Hopf algebra with product
$m_{\calD}$, coproduct $\Delta_{\calD}$, and unit element
$1_{\calD}$.

 A map  $f: \calE \to \calD$ of graded coalgebras is a
 \demph{connection} on $\calD$ if $\calE$ is a $\calD$--module coalgebra, $f$ is a map of
 $\calD$-module coalgebras, and $\calE$ is connected.
That is, $\calE$ is an associative (left or right) $\calD$-module
whose action (denoted $\star$) commutes with the coproducts, so that
$\Delta_{\calE}(e \star d) = \Delta_{\calE}(e) \star
\Delta_{\calD}(d)$, for  $e\in \calE$ and $d \in \calD$, \emph{and}
the coalgebra map $f$ is also a module map, so that for $e\in\calE$
and $d\in\calD$ we have
\[
  (f\otimes f )\,\Delta_{\calE}(e)\ =\ \Delta_{\calD}\, f(e)
    \qquad\mbox{and}\qquad
  f(e\star d) \ = \ m_{\calD} \,( f(e)\otimes d)\,.
\]

\begin{theorem}\label{cover_implies_Hopf}
  If $\calE$ is a connection on $\calD$, then $\calE$ is also a
  Hopf module and a comodule algebra over $\calD$.
  It is also a one-sided Hopf algebra with left-sided unit $\Blue{1_{\calE}}:=f^{-1}(1_{\calD})$ and
  left-sided antipode.
\end{theorem}

\begin{proof}
 Suppose $\calE$ is a right $\calD$-module.
 Define the product $m_{\calE}:\calE\otimes\calE \to \calE$ via the $\calD$-action:
 $m_{\calE}:=\star\circ(1\otimes f)$.
 The one-sided unit is $1_{\calE}$.
 Then $\Delta_\calE$ is an algebra map.
 Indeed, for $e,e' \in \calE$, we have
\[
  \Delta_\calE(e\cdot e') \ =\  \Delta_\calE(e\star f(e'))
  \ =\ \Delta_\calE e \star \Delta_\calD f(e')
  \ =\  \Delta_\calE e \star (f\otimes f)(\Delta_\calE e')
  \ =\  \Delta_\calE e \cdot \Delta_\calE e'\,.
\]
As usual, $\varepsilon_\calE$ is just projection onto $\calE_0$. The
unit $1_{\calE}$ is one-sided, since
\[
  e\cdot 1_{\calE}
  \ =\ e\star f(1_{\calE})
  \ =\ e\star f(f^{-1}(1_{\calD}))
  \ =\ e\star 1_{\calD}
  \ =\ e \,,
\]
but $1_\calE \cdot e = 1_\calE \star f(e)$ is not necessarily equal
to $e$. The antipode $S$ may be defined recursively to satisfy
$m_\calE(S\otimes \id)\Delta_\calE = \varepsilon_\calE$, just as for
graded bialgebras with two-sided units.

Define $\rho \colon\calE\to\calE\otimes\calD$ by $\rho := (1\otimes
f)\, \Delta_{\calE}$, which gives a coaction so that $\calE$ is a
Hopf module and a comodule algebra over $\calD$.
\end{proof}

%------------------------------------------
\subsection{Operads and operad modules}\label{sec: operad}

Composition of coalgebras is the same product used to define operads
internal to a symmetric monoidal category~\cite[App.
B]{AguMah:2010}. A \demph{monoid} in a category with a product
$\bullet$ is an object $\calD$ with a morphism $\gamma \colon
\calD\bullet\calD \to \calD$ that is associative. An \demph{operad}
is a monoid in the category of graded sets with an analog of the
composition product $\circ$ defined in Section~\ref{sec: cccc main
results}.

The category of connected graded coalgebras is a symmetric monoidal
category under the composition of coalgebras, $\circ$.
 A \demph{graded Hopf operad} $\calD$ is a monoid in the monoidal category of connected graded
 coalgebras and coalgebra maps, under the composition product.
 That is, $\calD$ is equipped with associative composition maps
\[
  \gamma\colon\calD\circ \calD \to \calD,
   \hbox{ \ obeying \ }\Delta_\calD\gamma(a)\ =\
    (\gamma\otimes\gamma)\left(\Delta_{\calD\circ\calD}(a)\right) \
   \hbox{ \ for all \ }a\in \calD\circ\calD\,.
\]

A \demph{graded Hopf operad module $\calE$} is an operad module over
$\calD$ and a graded coassociative coalgebra whose module action is
compatible with its coproduct. We denote the left and right action
maps by $\mu_l:\calD\circ \calE \to \calE$ and $\mu_r:\calE\circ
\calD \to \calE,$ obeying, e.g., $\Delta_\calE\mu_r(b) =
(\mu_r\otimes\mu_r)\Delta_{\calE\circ\calD} b$ for all
$b\in\calE\circ\calD$.

\begin{example}
 $\ysym$ is an operad in the category of vector spaces.
 The action of $\gamma$ on $\composeinline{F_{t_0}}{F_{t_n}}{F_t}$ grafts the
 trees $t_0,\dotsc,t_n$ onto the tree $t$ and, unlike in Example \ref{ex: painted},
 forgets which nodes of the resulting tree came from $t$.
 This is  associative in the appropriate sense.
 The same $\gamma$ makes $\ysym$ an operad in the category of connected graded coalgebras,
 making it a graded Hopf operad.
 Finally, operads are operad modules over themselves, so $\ysym$ is also graded
 Hopf operad module.
\end{example}

\begin{remark}
 Our graded Hopf operads differ from those of \cite{GetJon:1}, who defined a Hopf operad to
 be an operad of \emph{level coalgebras},
 where each component $\calD_n$ is a coalgebra.
\end{remark}

\begin{theorem}\label{thm: inthopf}
 A graded Hopf operad $\calD$ is also a Hopf algebra with product
\begin{gather}\label{eq: operad to algebra}
   a\cdot b\ :=\ \gamma( b \otimes \Delta^{(n)} a)
\end{gather}
 where $b\in \calD_n$ and $\Delta^{(n)}$ is the iterated coproduct from $\calD$ to
 $\calD^{\otimes(n+1)}$.
\end{theorem}

 It is possible to swap the roles of $a$ and $b$ on the right-hand side
 of~\eqref{eq: operad to algebra}.
 Our choice agrees with the product in $\ysym$ and $\csym$.
  In fact, the well-known Hopf algebra structures of $\ysym$ and $\csym$
 follow from their structure as  graded Hopf operads.

\begin{lemma}\label{lem:opmod}
  If $\calC$ is a graded coalgebra and $\calD$ is a graded Hopf operad, then
  $\calD \circ \calC$ is a (left) graded Hopf operad module and
  $\calC \circ \calD$ is a (right) graded Hopf operad module.
\end{lemma}

\begin{lemma}\label{lem:coal}
 A graded Hopf operad module over a graded Hopf operad
 is also a module coalgebra.
\end{lemma}

\begin{theorem}\label{thm:suffcover}
  Given a coalgebra map $\lambda \colon \calC \to \calD$ from a  connected graded coalgebra $\calD$
  to a graded Hopf operad $\calD$ , the maps
  $\gamma\circ(1\circ\lambda)\colon \calD\circ\calC\to \calD$ and
  $\gamma\circ(\lambda\circ 1)\colon \calD\circ\calC\to \calD$ give
  connections on $\calD$.
\end{theorem}

%%------------------------------------------
\subsection{Examples of module coalgebra connections}\label{sec: Hopf op examples}

Eight of the nine compositions of Example~\ref{sec: cccc examples}
are connections on one or both of the factors $\calC$ and $\calD$.

\begin{theorem}\label{thm:eightex}
  For $\calC\in\{\ssym,\ysym,\csym\}$, the coalgebra compositions $\calC\circ\csym$ and
  $\csym\circ\calC$ are connections on $\csym$.
  For $\calC \in \{\ssym,\ysym, \csym\}$, the coalgebra compositions $\calC\circ\ysym$ and
  $\ysym\circ\calC$ are connections on $\ysym$.
\end{theorem}

Note that $\csym\circ\ysym$ is a connection on both $\csym$ and on
$\ysym$, which gives two distinct one-sided Hopf algebra structures.
Similarly, $\ysym\circ\ysym$ is a connection on $\ysym$ in two
distinct ways (again leading to two distinct one-sided Hopf
structures).

%%%%%%%%%%%%%%%%%%%%%%%%%%%%%%%%%%%%%%%%%%%%%%%%%%%%%%%%%

%-------------------------------------------------------------------
% Begin SECTION
%-------------------------------------------------------------------
\section{Three Examples}\label{sec:Examples}

The three underlined algebras in Example~\ref{sec: cccc examples}
arose previously in algebra, topology, and category theory.

%------------------------------------------------------
\subsection{Painted Trees} \label{sec: painted}
 A \demph{painted binary tree} is a
planar binary tree $t$, together with a (possibly empty) upper order
ideal of its node poset. We indicate this ideal  by painting on top
of a representation of $t$. For example,
\[
    \includegraphics{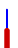}\, {,}   \quad\quad
    \includegraphics{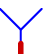}\!\! {,} \quad\quad
    \includegraphics{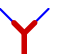} \!\! {,} \quad\quad
    \includegraphics{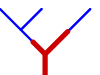} \!\! {,} \quad\quad
    \includegraphics{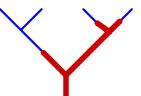} \!\!\! {.}
\]

An \demph{$A_n$-space} is a topological $H$-space with a weakly
associative multiplication of points (\cite{Sta:1963}). Maps between
$A_{n}$-spaces preserve the multiplicative structure only up to
homotopy. \cite{Sta:1970} described these maps combinatorially using
cell complexes called multiplihedra, while~\cite{BoaVog:1973}  used
spaces of painted trees. Both the spaces of trees and the cell
complexes are homeomorphic to convex polytope realizations of the
multiplihedra as shown in (\cite{forcey1}).

If $f\colon ( X,\Blue{\Placeholder}) \to (Y,\Red{\ast})$ is a map of
$A_n$-spaces, then the different ways to multiply and map $n$ points
of $X$ are represented by a painted tree. Unpainted nodes are
multiplications in $X$, painted nodes are multiplications in $Y$,
and the beginning of the painting indicates that $f$ is applied to a
given point in $X$,
%See Figure \ref{fig: maps to painted}.
%%%%%%%%%%%%%%%%%%%%%%%%%%%%%%%%%%%%%
%\begin{figure}[htb]
\[
    f(\Blue{a})\,\Red{\bm\ast}\,\bigl(f(\Blue{b\,\Placeholder\,c})\,
     \Red{\bm\ast}\,f(\Blue{d})\bigr)
    \ \longleftrightarrow \
    {\raisebox{-.5\height}{\includegraphics{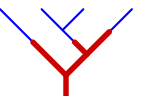}}}.
\]
%\caption{$A_n$-maps between $H$-spaces correspond to painted binary trees.}
%\label{fig: maps to painted}
%\end{figure}

%------------------------------------------
\subsubsection{Algebra structures on painted trees.}

Let $\p_n$ be the poset of painted trees on $n$ internal nodes, with
partial order inherited from the identification with $\m_{n+1}$.
%We show $\p_3$ in Figure~\ref{fig:_multi_paint}.
%%%%%%%%%%%%%%%%%%%%%%%%%%%%%%%%%%%%%%%%%%%%%%%%%%%%%%%%%%%%%%%%%%%%%%
%\begin{figure}[htb]
%
%\[
%  \begin{picture}(260,238)(2,2)
%   \put(2,5){\includegraphics[height=240pt]{M4.eps}}
%
%   \put(119,224){\includegraphics{p4321.s.eps}}%
%
%   \put(56,204){\includegraphics{p4312.s.eps}}
%   \put(182,203){\includegraphics{p3421.s.eps}}
%
%   \put(19,163){\includegraphics{p4213.s.eps}}
%   \put(74,163){\includegraphics{p4231.s.eps}}
%   \put(151,173){\includegraphics{p3412.s.eps}}
%   \put(224,164){\includegraphics{p2431.s.eps}}
%
%   \put( 4,116.5){\includegraphics{p4123.s.eps}}
%   \put(58,121){\includegraphics{p3214.s.eps}}
%   \put(105,130.7){\includegraphics{p3241.s.eps}}
%   \put(160.8,130.5){\includegraphics{p2413.s.eps}}
%   \put(235,118){\includegraphics{p1432.s.eps}}
%
%   \put(21, 70){\includegraphics{p3124.s.eps}}
%   \put(80.3, 74){\includegraphics{p2143.s.eps}}
%   \put(121, 84){\includegraphics{p2314.s.eps}}
%   \put(155.5, 90){\includegraphics{p2341.s.eps}}
%   \put(228, 72){\includegraphics{p1423.s.eps}}
%
%   \put(49, 32){\includegraphics{p2134.s.eps}}
%   \put(128.5, 38.5){\includegraphics{p1243.s.eps}}
%   \put(185, 24){\includegraphics{p1324.s.eps}}
%
%   \put(121, 0){\includegraphics{p1234.s.eps}}
%   \end{picture}
%\]
%\caption{The $1$-skeleton of the multiplihedron $\m_4$.}
%\label{fig:_multi_paint}
%\end{figure}
%%%%%%%%%%%%%%%%%%%%%%%%%%%%%%%%%%%%%%%%%%%%%%%%%%%%%%%%%%%%
%This picture shows that the
%vertices are the elements of a lattice whose Hasse diagram is the 1-skeleton
%of the polytope in the view shown.
The order on $\m_{n+1}$ is studied in~(\cite{FLS:1}).
%%(Note that the map from $\p_\bb$ to $\m_\bb$ actually lands in
%$\m_+$.)

We describe the key definitions of Section~\ref{sec: cccc main
results} and Section~\ref{sec: hopf} for
$\Blue{\psym}:=\ysym\circ\ysym$. In the fundamental basis $\bigl\{
F_p \mid p \in \p_\bb \bigr\}$ of $\psym$ , the counit is
$\varepsilon(F_p) =\delta_{0,|p|}$, and the product is given by
 \begin{gather*}
        \Delta(F_p)\ =\ \sum_{p \psplit (p_0,p_1)} F_{p_0} \otimes F_{p_1}\, ,
 \end{gather*}
 where the painting in $p\in\p_n$ is preserved in the splitting $p\psplit(p_0,p_1)$.

For example, we have
\[
  \Delta(F_{\;\includegraphics{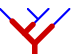}})\ = \
   1\otimes F_{\;\includegraphics{p2314.eps}}\ +\
   F_{\;\includegraphics{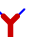}}\otimes F_{\;\includegraphics{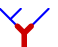}}
   \ +\ F_{\;\includegraphics{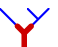}}\otimes F_{\;\includegraphics{p12.eps}}
   \ +\ F_{\;\includegraphics{p2314.eps}}\otimes 1\,.
\]

The identity map on $\ysym$ makes $\psym$ into a connection on
$\ysym$. By Theorem~\ref{cover_implies_Hopf}, $\psym$ is thus also a
one-sided Hopf algebra, a $\ysym$-Hopf module, and a
$\ysym$-comodule algebra. The product $F_p\cdot F_q$ in $\psym$ as
 \[
    F_p\cdot F_q\  =\  \sum_{p\psplit (p_0,p_1,\dots,p_r)} F_{(p_0,p_1,\dots,p_r)/q^+} \,,
 \]
where the painting in $p$ is preserved in the splitting
$(p_0,p_1,\dots,p_r)$, and $q^+$ signifies that $q$ is painted
completely before grafting. For example,
\[
     F_{\;\includegraphics{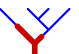}}\cdot
     F_{\;\includegraphics{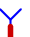}}
     \ =\
     F_{\;\includegraphics{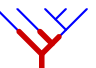}}\ +\
     F_{\;\includegraphics{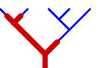}}\ +\
     F_{\;\includegraphics{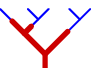}}\ +\
     F_{\;\includegraphics{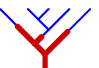}}\,.
\]

The painted tree \includegraphics{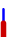} with $0$ nodes is a right
multiplicative identity element,
\[
    F_{\;\includegraphics{p0.eps}} \cdot F_q\ =\
   F_{q^+} \qquad\hbox{and}\qquad F_q\cdot F_{\;\includegraphics{p0.eps}}\ =\
  F_q \qquad \mbox{for } q\in \p_\bb\,.
\]

As $\psym$ is graded and connected, it has an antipode.

\begin{theorem}\label{thm: painted is hopf}
 There are unit and antipode maps $\mu\colon \K \to \psym$ and
 $S \colon \psym \to \psym$ making $\psym$ a one-sided Hopf algebra.
\end{theorem}

The $\ysym$-Hopf module structure on $\psym$ from
Theorem~\ref{cover_implies_Hopf} has  coaction
 \[
    \rho(F_p)\  =\  \sum_{p \psplit (p_0,p_1)} F_{p_0} \otimes F_{f(p_1)}\,,
 \]
where the painting in $p$ is preserved in $p_0$ amd forgotten in
$p_1$.

Since painted trees and bi-leveled trees both index vertices of the
multiplihedra, these structures for $\psym$ give structures on the
linear span $\msym_+$ of bi-leveled trees with at least one node.

\begin{corollary}\label{thm: msym+ is Hopf}
 The $\ysym$ action and coaction defined in \cite[Section 4.1]{FLS:1}
 make $\msym_+$ into a Hopf module isomorphic to the Hopf module
 $\psym$. \hfill \qed
\end{corollary}

%
%%%%%%%%%%%%%%%%%%%%%%%%%%%%%%%%%%%%%%%%%%%%%%%%%%%%%%%%%
\subsection{Composite Trees}\label{sec: weighted}

In a forest of combs attached to a binary tree, the combs may be
replaced by corollae  or by a positive \demph{weight} counting the
number of leaves in the comb. These all give \demph{composite
trees}.
 \begin{equation}\label{Eq:composite_trees}
  \raisebox{-17pt}{${\displaystyle
   \includegraphics{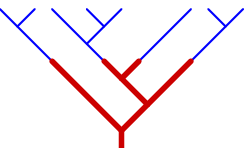} \quad\raisebox{17pt}{=}\quad
   \includegraphics{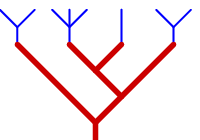} \quad\raisebox{17pt}{=}\quad
   \begin{picture}(50,37.5)(1,-3)
    \put(3,-3){\includegraphics{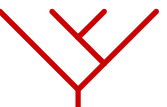}}
    \put( 0,28){$2$}    \put(15,28){$3$}
    \put(30,28){$1$}    \put(45,28){$2$}
   \end{picture}}$}
 \end{equation}
Composite trees with weights summing to $n{+}1$, $\ck_n$, were shown
to be the vertices of a $n$-dimensional polytope, the
\demph{composihedron}, $\ck(n)$~(\cite{forcey2}). This sequence of
polytopes is used to parameterize homotopy maps between strictly
associative and homotopy associative $H$-spaces. For small values of
$n$, the polytopes $\ck(n)$ also appear as the commuting diagrams in
enriched bicategories~(\cite{forcey2}). These diagrams appear in the
definition of pseudomonoids~\cite[App. C]{AguMah:2010}.

%%%%%%%%%%%%%%%%%%%%%%%%%%%%%%%%%%%%%%%%%%%%%%%%
\subsubsection{Algebra structures on composite trees}

We describe the key definitions of Section~\ref{sec: cccc main
results} and Section~\ref{sec: hopf} for
$\Blue{\cksym}:=\ysym\circ\csym$. In the fundamental basis $\bigl\{
F_p \mid p \in \ck_\bb \bigr\}$ of $\cksym$, the counit is
$\varepsilon(F_p)=\delta_{0,|p|}$ and the coproduct is
 \begin{gather*}
        \Delta(F_p)\ =\ \sum_{p \psplit (p_0,p_1)} F_{p_0} \otimes F_{p_1}\,,
 \end{gather*}
 where the painting in $p\in\ck_\bb$ is preserved in the splitting $p\psplit(p_0,p_1)$.

Here is an example, written in terms of the weighted trees.
\[
   \Delta(F_{\CTIT{2}{1}{2}})\ =\
    F_{\CTO{1}}\otimes F_{\CTIT{2}{1}{2}} \ +\
    F_{\CTO{2}}\otimes F_{\CTIT{1}{1}{2}} \ +\
    F_{\CTI{2}{1}}\otimes F_{\CTI{1}{2}} \ +\
    F_{\CTIT{2}{1}{1}}\otimes F_{\CTO{2}}  \ +\
    F_{\CTIT{2}{1}{2}}\otimes F_{\CTO{1}} \,.
\]

For the product,  Theorem~\ref{cover_implies_Hopf} using the left
module coalgebra action defined in Lemma~\ref{lem:coal} gives
\[
   F_a \cdot F_b\ :=\ g(F_a)\star F_b\,,\qquad
   \mbox{where }a,b\in\ck_\bb\,,
\]
where $g\colon\cksym\to\csym$ is the connection. On the indices, it
sends a composite tree $a$ to the unique comb $g(a)$ with the same
number of nodes as $a$. For the action $\star$, $g(a)$ is split in
all ways to make a forest of $|b|{+}1$ combs, which are grafted onto
the leaves of the forest of combs in $b$, then each tree in the
forest is combed and attached to the binary tree in $b$. We
illustrate one term in the product. Suppose that
\newline
   $a=$
   \raisebox{-3pt}{\begin{picture}(14,18)\put(1.5,0){\includegraphics[height=12pt]{1.d.eps}}
    \put(0,13.5){\small 2}\put(11,13.5){\small 1}\end{picture}}
   $ = \raisebox{-3pt}{\includegraphics{p213.d.eps}}$
     and  $b = $
   \raisebox{-3pt}{\begin{picture}(22,23)\put(2,0){\includegraphics[height=16pt]{12.d.eps}}
    \put(0,17.5){\small 1}\put(10,17.5){\small 2}\put(20,17.5){\small 1}\end{picture}}
    $= $
   \raisebox{-3pt}{\includegraphics[height=22pt]{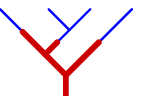}}.
   Then $g(a) = $
   \includegraphics{21.d.eps}\ .
    One way to split $g(a)$ gives the forest
    $ \bigl(\, \raisebox{-3pt}{\includegraphics{0.d.eps}\;,\;
      \includegraphics{1.d.eps}\;,\;
      \includegraphics{0.d.eps}\;,\;
      \includegraphics{1.d.eps}}\;\bigr)$.
   Graft this onto $b$ to get
   \raisebox{-3pt}{\includegraphics{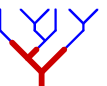}},
   then comb the forest to get
   \raisebox{-3pt}{\includegraphics{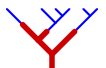}},
   which is
   \raisebox{-3pt}{\begin{picture}(22,23)\put(2,0){\includegraphics[height=16pt]{12.d.eps}}
    \put(0,17.5){\small 1}\put(10,17.5){\small 3}\put(20,17.5){\small 2}\end{picture}}.\vspace{2pt}
   Doing this for the other nine splittings of $g(a)$ gives,
\[
    F_{\CTI{2}{1}}\cdot F_{\CTIT{1}{2}{1}}\ =\
    F_{\CTIT{3}{2}{1}}\ +\   3  F_{\CTIT{1}{4}{1}}\ +\
    F_{\CTIT{1}{2}{3}}\ +\   2  F_{\CTIT{2}{3}{1}}\ +\
    F_{\CTIT{2}{2}{2}}\ +\    2 F_{\CTIT{1}{3}{2}}\,.
\]

%%%%%%%%%%%%%%%%%%%%%%%%%%%%%%%%%%%%%%%%%%%%%%%%%%%%%%%%%
\subsection{Composition trees}\label{sec: simplices}

The simplest composition of Fig.~\ref{fig: diamond} is
$\csym\circ\csym$, whose basis is indexed by combs over combs. If we
represent these as weighted trees as in~\eqref{Eq:composite_trees},
we see that we may identify combs over combs with $n$ internal nodes
as compositions of $n{+}1$. Thus we refer to these as
\demph{composition trees}.
\[
   \includegraphics{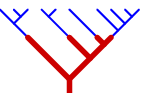}\
   \raisebox{10pt}{$\Longleftrightarrow$}\
   \begin{picture}(28,24)
    \put(2.5,0){\includegraphics{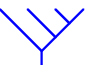}}
    \put( 0,17.5){\small 3} \put( 8,17.5){\small 2}
    \put(15.5,17.5){\small 1} \put(24,17.5){\small 4}
   \end{picture}\
   \raisebox{10pt}{$\Longleftrightarrow  \  (3,2,1,4)$\,.}
\]
The coproduct is again given by splitting. Since $(1,3)$ has the
four splittings,
 \begin{equation}\label{Eq:comp_split}
   \raisebox{-3pt}{\includegraphics{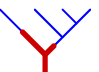}}\
   \raisebox{3pt}{$\xrightarrow{\ \curlyvee\ }$}\
   \Bigl(\;\raisebox{-3pt}{\includegraphics{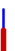}}\,,
         \raisebox{-3pt}{\includegraphics{p3421.s.eps}}\Bigr)
   \,,\quad
   \Bigl(\raisebox{-3pt}{\includegraphics{p12.d.eps}}\,,
         \raisebox{-3pt}{\includegraphics{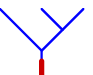}}\Bigr)
    \,,\quad
   \Bigl(\raisebox{-3pt}{\includegraphics{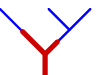}}\,,
         \raisebox{-3pt}{\includegraphics{p21.d.eps}}\Bigr)
    \,,\quad
   \Bigl(\raisebox{-3pt}{\includegraphics{p3421.s.eps}}\,,
         \raisebox{-3pt}{\includegraphics{p0.d.eps}}\;\Bigr)\,,
 \end{equation}
we have $\Delta(F_{1,3}) = F_{1}\otimes F_{1,3}+F_{1,1}\otimes
F_{3}+
   F_{1,2}\otimes F_{2}+ F_{1,3}\otimes F_{1}$.

As we remarked, there are two connections  $\csym\circ\csym \to
\csym$, using either the right or left action of $\csym$. This gives
two new one-sided Hopf algebra structures on compositions. With the
right action, we have $F_{1,3}\cdot
F_2=2F_{1,1,3}+F_{1,2,2}+F_{1,3,1}$, as
 \begin{equation}\label{Eq:prodCC}
   F_{\;\includegraphics{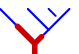}} \cdot   F_{\;\includegraphics{p21.eps}}
   \ =\
   F_{\;\includegraphics{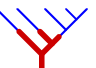}}\ +\ F_{\;\includegraphics{p35421.eps}}\ +\
   F_{\;\includegraphics{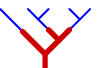}}\ +\ F_{\;\includegraphics{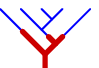}}\,,
 \end{equation}
which may be seen by grafting the different
splittings~\eqref{Eq:comp_split} onto the tree
$\includegraphics{p21.eps}$ and coloring
$\includegraphics{p21.eps}$.

\cite{ForSpr:2010} defined a one-sided Hopf algebra $\Delta{S}ym$ on
the graded vector space spanned by the faces of the simplices. Faces
of the simplices correspond to subsets of $[n]$. Here is an example
of the coproduct of the basis element corresponding to
$\{1\}\subset[4],$ where subsets of $[n]$ are illustrated as circled
subsets of the circled edgeless graph on $n$ nodes numbered left to
right:
\[
  \Delta\raisebox{-3pt}{\includegraphics{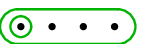}}\ =\
  \raisebox{-3pt}{\includegraphics{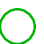}} \otimes
  \raisebox{-3pt}{\includegraphics{1.4.g.eps}} \ +\
  \raisebox{-3pt}{\includegraphics{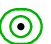}} \otimes
  \raisebox{-3pt}{\includegraphics{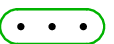}} \ +\
  \raisebox{-3pt}{\includegraphics{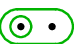}} \otimes
  \raisebox{-3pt}{\includegraphics{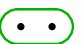}} \ +\
  \raisebox{-3pt}{\includegraphics{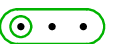}} \otimes
  \raisebox{-3pt}{\includegraphics{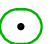}} \ +\
  \raisebox{-3pt}{\includegraphics{1.4.g.eps}} \otimes
  \raisebox{-3pt}{\includegraphics{0.0.g.eps}}
\]
Here is an example of the product
\[
  \raisebox{-3pt}{\includegraphics{0.1.g.eps}}\cdot
  \raisebox{-3pt}{\includegraphics{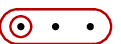}}\ =\
  \raisebox{-3pt}{\includegraphics{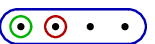}}\ +\
  \raisebox{-3pt}{\includegraphics{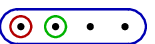}}\ +\
  \raisebox{-3pt}{\includegraphics{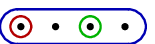}}\ +\
  \raisebox{-3pt}{\includegraphics{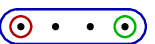}}\,.
\]

Let $\varphi$ denote the bijection between subsets $S =\{
a,b,\dots\,c\} \subset[n]$ and compositions
$\varphi(S)=(a,b-a,\dots,n+1-c)$ of $n+1$. Applying this bijection
the indices of their fundamental bases gives a linear isomorphism
$\bvarphi\colon\Delta{S}ym \xrightarrow{\sim}\csym\circ\csym$, which
is nearly an isomorphism of one-sided Hopf algebras, as may be seen
by comparing these schematics of operations in $\Delta{S}ym$ to
formulas~\eqref{Eq:comp_split} and~\eqref{Eq:prodCC} in
$\csym\circ\csym$.

\begin{theorem}
 The map $\bvarphi$ is an isomorphism of coalgebras and an anti-isomorphism
 ($\bvarphi(a\cdot b)=\bvarphi(a)\cdot\bvarphi(b)$) of
 one-sided algebras.
\end{theorem}

\begin{corollary}\label{thm: F-S is cofree}
 The one-sided Hopf algebra of simplices introduced in (\cite{ForSpr:2010}) is cofree as a coalgebra.
\end{corollary}
%%%%%%%%%%%%%%%%%%%%%%%%%%%%%%%%%%%%%%%%%%%%%%%%%%%%%%%%%


\newpage
\begin{thebibliography}{16}
\providecommand{\natexlab}[1]{#1}
\providecommand{\url}[1]{\texttt{#1}} \expandafter\ifx\csname
urlstyle\endcsname\relax
  \providecommand{\doi}[1]{doi: #1}\else
  \providecommand{\doi}{doi: \begingroup \urlstyle{rm}\Url}\fi

\bibitem[Aguiar and Mahajan(2010)]{AguMah:2010}
M.~Aguiar and S.~Mahajan.
\newblock \emph{Monoidal functors, species and {H}opf algebras}, volume~29 of
  \emph{CRM Monograph Series}.
\newblock American Mathematical Society, Providence, RI, 2010.\vspace{-1pt}

\bibitem[Aguiar and Sottile(2005)]{AguSot:2005}
M.~Aguiar and F.~Sottile.
\newblock Structure of the {M}alvenuto-{R}eutenauer {H}opf algebra of
  permutations.
\newblock \emph{Adv. Math.}, 191\penalty0 (2):\penalty0 225--275, 2005.\vspace{-1pt}

\bibitem[Aguiar and Sottile(2006)]{AguSot:2006}
M.~Aguiar and F.~Sottile.
\newblock Structure of the {L}oday-{R}onco {H}opf algebra of trees.
\newblock \emph{J. Algebra}, 295\penalty0 (2):\penalty0 473--511, 2006.\vspace{-1pt}

\bibitem[Boardman and Vogt(1973)]{BoaVog:1973}
J.~M. Boardman and R.~M. Vogt.
\newblock \emph{Homotopy invariant algebraic structures on topological spaces}.
\newblock Lecture Notes in Mathematics, Vol. 347. Springer-Verlag, Berlin,
  1973.\vspace{-1pt}

\bibitem[Forcey(2008{\natexlab{a}})]{forcey1}
S.~Forcey.
\newblock Convex hull realizations of the multiplihedra.
\newblock \emph{Topology Appl.}, 156\penalty0 (2):\penalty0 326--347,
  2008{\natexlab{a}}.\vspace{-1pt}

\bibitem[Forcey(2008{\natexlab{b}})]{forcey2}
S.~Forcey.
\newblock Quotients of the multiplihedron as categorified associahedra.
\newblock \emph{Homology, Homotopy Appl.}, 10\penalty0 (2):\penalty0 227--256,
  2008{\natexlab{b}}.\vspace{-1pt}

\bibitem[Forcey and Springfield(2010)]{ForSpr:2010}
S.~Forcey and D.~Springfield.
\newblock Geometric combinatorial algebras: cyclohedron and simplex, 2010.
\newblock \emph{J. Alg. Combin.}, DOI: 10.1007/s10801-010-0229-5.\vspace{-1pt}

\bibitem[Forcey et~al.(2010)Forcey, Lauve, and Sottile]{FLS:1}
S.~Forcey, A.~Lauve, and F.~Sottile.
\newblock Hopf structures on the multiplihedra.
\newblock \emph{SIAM J. Discrete Math.}, DOI: 10.1137/090776834.\vspace{-1pt}

\bibitem[Getzler and Jones()]{GetJon:1}
E.~Getzler and J.~Jones.
\newblock Operads, homotopy algebra and iterated integrals for double loop
  spaces.
\newblock preprint, arXiv:hep-th/9403055.\vspace{-1pt}

\bibitem[Loday and Ronco(1998)]{LodRon:1998}
J.-L. Loday and M.~O. Ronco.
\newblock Hopf algebra of the planar binary trees.
\newblock \emph{Adv. Math.}, 139\penalty0 (2):\penalty0 293--309, 1998.\vspace{-1pt}

\bibitem[Malvenuto and Reutenauer(1995)]{MalReu:1995}
C.~Malvenuto and C.~Reutenauer.
\newblock Duality between quasi-symmetric functions and the {S}olomon descent
  algebra.
\newblock \emph{J. Algebra}, 177\penalty0 (3):\penalty0 967--982, 1995.\vspace{-1pt}

\bibitem[Montgomery(1993)]{Mont:1993}
S.~Montgomery.
\newblock \emph{Hopf algebras and their actions on rings}, volume~82 of
  \emph{CBMS Regional Conference Series in Mathematics}.
\newblock Conference Board of the Mathematical Sciences,
  Washington, DC, 1993.\vspace{-1pt}

\bibitem[Sloane()]{Slo:oeis}
N.~J.~A. Sloane.
\newblock The on-line encyclopedia of integer sequences.
\newblock published electronically at {\tt
  www.research.att.com/$\scriptstyle\sim$njas/sequences/}.\vspace{-1pt}

\bibitem[Stasheff(1963)]{Sta:1963}
J.~Stasheff.
\newblock Homotopy associativity of {$H$}-spaces. {I}, {II}.
\newblock \emph{Trans. Amer. Math. Soc. {\bf108} (1963), 275--292; ibid.},
  108:\penalty0 293--312, 1963.\vspace{-1pt}

\bibitem[Stasheff(1970)]{Sta:1970}
J.~Stasheff.
\newblock \emph{{$H$}-spaces from a homotopy point of view}.
\newblock Lecture Notes in Mathematics, Vol. 161. Springer-Verlag, Berlin,
  1970.\vspace{-1pt}

\bibitem[Takeuchi(1971)]{Tak:71}
M.~Takeuchi.
\newblock Free {H}opf algebras generated by coalgebras.
\newblock \emph{J. Math. Soc. Japan}, 23:\penalty0 561--582, 1971.
\newpage

\end{thebibliography}
\end{document}